\documentclass[a4paper]{amsart}
\usepackage[english]{babel}
\usepackage[utf8]{inputenc}
\usepackage[T1]{fontenc}
\usepackage{lmodern}

\usepackage[shortlabels]{enumitem}
\setlist[enumerate]{label={\arabic*.}}
\setlist[description]{font=\normalfont\itshape}

\usepackage[dvipsnames]{xcolor}
\usepackage[backref]{hyperref}
\hypersetup{colorlinks=true,linkcolor=Maroon,citecolor=OliveGreen,urlcolor=Maroon}

\usepackage{cleveref}
\usepackage{microtype}

\usepackage{setspace}
\usepackage{amsmath,amsfonts,amsthm,mathtools,commath,amssymb}

\usepackage{todonotes}

\def\N{\mathbf{N}}
\def\Z{\mathbf{Z}}
\def\Q{\mathbf{Q}}

\def\C{\mathbf{C}}
\def\F{\mathbf{F}}
\def\PP{\mathbf{P}}
\def\Pr{\mathbf{P}}

\def\basis{\mathcal{B}}
\def\steinbergbasis{\mathcal{S}}
\def\integers{\mathcal{O}}

\DeclareMathOperator{\lin}{Lin}
\DeclareMathOperator{\Tr}{Tr}

\DeclareMathOperator{\SL}{SL}

\DeclareMathOperator{\slfrak}{\mathfrak{sl}}
\DeclareMathOperator{\spfrak}{\mathfrak{sp}}
\DeclareMathOperator{\sofrak}{\mathfrak{so}}
\DeclareMathOperator{\sufrak}{\mathfrak{su}}
\DeclareMathOperator{\liealg}{\mathfrak{g}}
\DeclareMathOperator{\g}{\mathfrak{g}}
\DeclareMathOperator{\gfrak}{\mathfrak{g}}
\DeclareMathOperator{\ffrak}{\mathfrak{f}}
\DeclareMathOperator{\efrak}{\mathfrak{e}}

\DeclareMathOperator{\diam}{diam}
\DeclareMathOperator{\tower}{T}
\DeclareMathOperator{\PSL}{PSL}

\DeclareMathOperator{\Frob}{Frob}
\DeclareMathOperator{\Gal}{Gal}
\DeclareMathOperator{\irr}{(irr)}
\DeclareMathOperator{\density}{density}
\DeclareMathOperator{\image}{im}
\DeclareMathOperator{\fieldnorm}{Norm}
\DeclareMathOperator{\ad}{ad}

\DeclarePairedDelimiter\floor{\lfloor}{\rfloor}
\DeclarePairedDelimiter{\ceil}{\lceil}{\rceil}

\renewcommand{\abs}[1]{\lvert#1\rvert}

\newtheorem{theorem}{Theorem}[section]
\newtheorem*{theorem*}{Theorem}
\newtheorem{lemma}[theorem]{Lemma}
\newtheorem{proposition}[theorem]{Proposition}
\newtheorem{corollary}[theorem]{Corollary}
\newtheorem*{example*}{Example}

\theoremstyle{definition}

\newtheorem{remark}[theorem]{Remark}

\newtheorem{example}[theorem]{Example}
\newtheorem{conjecture}[theorem]{Conjecture}
\newtheorem*{conjecture*}{Conjecture}

\begin{document}
\baselineskip=13pt 
\def\UrlBreaks{\do\/\do-}

\title{Diameter bounds for finite simple Lie algebras}
\author{}
\date{}

\author{Marco Barbieri}
\address{Marco Barbieri, Faculty of Mathematics and Physics, University of Ljubljana, Jadranska 19, 1000 Ljubljana, Slovenia}
\email{marco.barbieri@fmf.uni-lj.si}

\author{Urban Jezernik}
\address{Urban Jezernik, Faculty of Mathematics and Physics, University of Ljubljana, Jadranska 19, 1000 Ljubljana, Slovenia / Institute of Mathematics, Physics, and Mechanics, Jadranska 19, 1000 Ljubljana, Slovenia}
\email{urban.jezernik@fmf.uni-lj.si}

\author{Matevž Miščič}
\address{Matevž Miščič, Faculty of Mathematics and Physics, University of Ljubljana, Jadranska 19, 1000 Ljubljana, Slovenia / Institute of Mathematics, Physics, and Mechanics, Jadranska 19, 1000 Ljubljana, Slovenia}
\email{matevz.miscic@imfm.si}

\thanks{This work has been supported by the Slovenian Research Agency program P1-0222 and grants J1-50001, J1-4351, J1-3004, N1-0217.}

\begin{abstract}
    We prove strong and explicit diameter bounds for finite simple Lie algebras, which parallel Babai's conjecture for finite simple groups. Specifically, we show that any nonabelian finite simple Lie algebra $\g$ over $\F_p$ has diameter $O((\log{\abs{\g}})^D)$ for $D \approx 3.11$ with respect to any generating set. For $\F_p$-forms of classical Lie algebras of fixed Lie type, we establish the sharper bound $O(\log{\abs{\g}})$ when the generators are chosen uniformly at random.
\end{abstract}

\maketitle


\section{Introduction}

\subsection{Diameters of finite simple groups}
Let $G$ be a group generated by a set $A$. The \emph{diameter} of $G$ with respect to $A$, denoted $\diam(G,A)$, is the smallest integer $k$ such that every element of $G$ can be written as a product of at most $k$ elements from $A$. In effect, it measures how quickly $A$ generates $G$. Diameters have been widely studied in finite simple groups, led by the following conjecture.

\begin{conjecture}[rapid generation of groups -- Babai's conjecture \cite{babais-conjecture}]
There exist absolute constants $C, D$\footnote{A stronger version of the conjecture predicts that one might take $D = 2$.} so that the following holds. Let $G$ be a nonabelian finite simple group and $A$ any generating set of $G$. Then $\diam(G,A) \leq C(\log \abs{G})^D$.
\end{conjecture}

The conjecture is open in general. Being asymptotic, it reduces to proving the existence of such constants $C, D$ for the infinite families of alternating groups ${\rm Alt}(n)$ (as $n \to \infty$) and finite simple groups of Lie type such as $\PSL_n(\F_q)$ (as $n \to \infty$ and/or $q \to \infty$). The conjecture is known to hold in the special case of finite simple groups of fixed Lie type (for example $\PSL_n(\F_q)$ with $n$ fixed and $q \to \infty$) \cite{helfgott2008growth,breuillard2011approximate,pyber2016growth,bajpai2024new}. What all these proofs have in common is the use of a \emph{product theorem}, which states that generating sets of finite simple groups of fixed Lie type grow uniformly under multiplication (unless they are already very large).

\subsection{Diameters of finite simple Lie algebras}

The same game can be played in Lie algebras. Let $\gfrak$ be a Lie algebra over $\F_p$. For subsets $X,Y \subseteq \gfrak$ write
\[
    X + Y = \{x + y \mid x \in X, \ y \in Y\} 
    \quad \text{and} \quad 
    [X, Y] = \{[x, y] \mid x \in X, \ y \in Y\}.
\]
Suppose that $\gfrak$ is generated as a Lie algebra by a set $A$. Let
\[
    A^1 = \{0\} \cup A, \qquad
    A^k = \bigcup_{0 < j < k} \left((A^j + A^{k-j}) \cup [A^j, A^{k-j}]\right) \quad \text{for } k \geq 2.
\]
The \emph{diameter} of $\gfrak$ with respect to $A$, denoted $\diam(\gfrak, A)$, is the least $k$ such that $A^k = \gfrak$. We can now state the linear version of Babai's conjecture as follows.

\begin{conjecture}[rapid generation of Lie algebras]
There exist absolute constants $C, D$ so that the following holds. Let $\gfrak$ be a nonabelian finite simple Lie algebra over $\F_p$ and $A$ any generating set of $\gfrak$. Then $\diam (\gfrak, A) \leq C (\log \abs{\gfrak})^D$.
\end{conjecture}

How are the two conjectures related? Classical Lie groups are intimately connected to their Lie algebras. This correspondence is facilitated by the Baker-Campbell-Hausdorff formula, which behaves well over local fields but much less so over finite fields $\F_p$.\footnote{In fact, a correspondence between finite $p$-groups and finite Lie algebras over $\F_p$ exists when the groups are nilpotent of class less than $p$, known as the Lazard correspondence \cite[Chapter 10]{khukhro1998p}.} While an $\exp$--$\log$ correspondence exists at the level of generators, there is no global link between finite simple groups of Lie type and finite simple Lie algebras over $\F_p$. Nevertheless, studying the linear Lie algebra setting captures aspects of the complexity present in the group-theoretic problem.

By a recent result of Dona \cite{dona-sum-bracket}, rapid generation holds for classical Lie algebras of fixed Lie type (as for groups) that are split over $\F_p$. These are exactly the Lie algebras obtained from complex simple Lie algebras via base change to $\F_p$.\footnote{There are more simple algebras over $\F_p$ than just the classical ones and their forms. See \Cref{sec:classical_lie_algebras} for details over an algebraically closed field of positive characteristic.} The constant $D$ in rapid generation can be taken as $O(\dim(\liealg)^2 \log \dim(\liealg))$, remaining fixed for a given Lie type. The proof in fact uses methods from the group case (in particular, a \emph{sum-bracket theorem}), but in a simplified, linear form.

\subsection{Contributions}

In this paper, we reuse methods developed by Dona and build on them to fully prove linear Babai's conjecture with explicit $D$. We study classical Lie algebras in detail and provide even stronger diameter bounds when the Lie type is fixed and generators are chosen at random.

\subsubsection{Rapid generation of Lie algebras}

\begin{theorem}
    \label[theorem]{thm:babai_conjecture}
    Let $D = \log 2 / \log(5/4)$. For any $\epsilon > 0$ there is a constant $C$ such that the following holds. Let $\gfrak$ be a nonabelian finite simple Lie algebra over $\F_p$ and $A$ any generating set of $\gfrak$. Then
    \[
        \diam(\liealg, A) \leq C(\log\abs{\liealg})^{D + \epsilon}.
    \]
\end{theorem}

Notably, finite simple Lie algebras are not fully classified, unlike finite simple groups. Our argument bypasses this by relying only on simplicity. Moreover, the proof yields the more explicit bound $O_{\epsilon}\left((\log p)^{D+\epsilon}(\dim \g)^2\right)$. When $p$ is fixed, this becomes 
$O_{p}\left((\log \abs{\g})^2\right)$, matching the stronger version of Babai's conjecture for finite simple groups. In general, our exponent $D \approx 3.11$ is slightly larger.

\begin{remark}
    The result above is also valid for nonabelian finite simple Lie algebras $\gfrak$ over $\F_q$ with $q$ a power of $p$, as long as $\gfrak$ has finite diameter with respect to $A$. This is not the same as saying that $A$ generates $\gfrak$ as a Lie algebra over $\F_q$, since the sets $A^k$ are entirely contained in the $\F_p$-Lie algebra generated by $A$, which might be a proper subalgebra of $\gfrak$ when $q$ is a proper power of $p$. However, the simplicity of $\gfrak$ as an $\F_q$-Lie algebra implies its simplicity as an $\F_p$-Lie algebra.\footnote{The Lie bracket on $\gfrak$ is $\F_q$-bilinear, so if $I$ is an $\F_p$-ideal of $\gfrak$, then $[I, \gfrak]$ is an $\F_q$-ideal contained in $I$, hence $[I, \gfrak] = \gfrak$ and so $I = \gfrak$.} Consequently, if $A$ generates $\gfrak$ as a Lie algebra over $\F_p$, the same polylogarithmic diameter bound applies.
\end{remark}

The method behind the result incorporates ideas from Dona's paper, though our implementation differs. To outline the main steps, define for any subset $X \subseteq \g$,
\[
\ell(X) = \max \{ \abs{X \cap L} \mid L \text{ is a one-dimensional subspace of } \g \}.
\]
Thus, $\ell(X)$ is the maximum number of elements from $X$ on any line in $\g$. We first show that if a generating set $A$ of $\gfrak$ satisfies $\ell(A^k) = p$ for some $k$, then the entire Lie algebra can be covered in a modest number of extra steps. 

\begin{proposition}
    \label[proposition]{prop:full_line}
    Let $\gfrak$ be a nonabelian finite simple Lie algebra over $\F_p$ and $A$ any generating set of $\gfrak$. Assume that $\ell(A^k) = p$ for some $k$. Then $\diam(\gfrak, A) \leq k \dim \gfrak + (\dim \gfrak)^2$.
\end{proposition}

This proves rapid generation for Lie algebras when $p$ is fixed. Starting from a generating set $A$, we have $\ell(A^p)= p$, so $\diam(\gfrak, A) \le O_p\left((\dim \gfrak)^2\right) = O_p\left( (\log \abs{\gfrak})^2\right)$.\footnote{For finite simple groups, rapid generation in groups of Lie type over bounded fields and high rank is considerably more challenging and is only known to hold for random generators \cite{eberhard2022babai}.}

To cover the case when $p$ varies, we show that $\ell(A^k) = p$ (or is at least sufficiently large) can be achieved with a small value of $k$. We do this by proving a \emph{product theorem in lines}: if a generating set $A$ does not yield a large enough $\ell(A^k)$, then a few extra steps induce uniform growth in $\ell(A^k)$. Here is the precise statement.

\begin{theorem}[product theorem in lines]
\label[theorem]{line_growth}
For every $\epsilon > 0$ there is a constant $C > 0$ such that for any generating set $A$ of any nonabelian finite simple Lie algebra $\g$ over $\F_p$ and any $k \geq C$, we have
\[
\ell(A^{2k + \dim \gfrak}) \geq \ell(A^k)^{5/4 - \epsilon}
\quad \text{or} \quad
\ell(A^{C(2k + \dim \gfrak)}) = p.
\]
\end{theorem}

One-dimensional growth as it appears here was already leveraged by Dona to prove the sum-product theorem in Lie algebras. We extract key elements of their argument and bootstrap it to attain faster growth as given by the theorem, and this eventually leads to the diameter bounds stated above. The constants $2$ and $5/4$ in the theorem arise from applying a sum-product theorem in finite fields \cite{mohammadi2023attaining}.\footnote{Having a stronger sum-product theorem would improve the value of $D$. See Remark~\ref{remark on stronger sum product theorem}.}

\subsubsection{Extremely rapid random generation of classical Lie algebras}

We study classical Lie algebras over $\F_p$ in detail. Every Lie algebra over $\F_p$ that becomes isomorphic to a complex simple Lie algebra base changed to the algebraic closure of $\F_p$ arises (for $p > 3$) from an \emph{$\F_p$-form of a classical Lie algebra}. These come in two flavors:
\begin{description}
    \item[split] obtained from complex simple Lie algebras by base change to $\F_p$, \emph{e.g.}, $\slfrak_n(\F_p)$;
    \item[non-split] arising from symmetries in the Dynkin diagram, \emph{e.g.}, $\sufrak_n(\F_{p^2})$. 
\end{description}
We prove that both types exhibit extremely rapid random generation, although in the non-split case we must exclude a small density subset of primes.

\begin{theorem}
    Let $\gfrak$ be an $\F_p$-form of a classical Lie algebra. 
    
    \begin{enumerate}[leftmargin=*]
    \item 
    If $\gfrak$ is split, then there exists a constant $C > 0$, depending only on the Lie type of $\gfrak$, such that for a uniformly random pair of elements $X, Y \in \gfrak$, we have
    \[
        \diam(\gfrak, \{X, Y\}) \leq C \log \abs{\gfrak}
    \]
    with probability tending to $1$ as $p$ tends to infinity.
    
    \item If $\gfrak$ is non-split, then for every $\epsilon > 0$ there is a set of primes of density at least $1 - \epsilon$ and a constant $C_{\epsilon} > 0$, depending only on the Lie type of $\gfrak$ and $\epsilon$, such that the same conclusion as above holds along primes in this set.
    \end{enumerate}
\end{theorem}

A specific instance of this phenomenon was observed for the Lie algebra $\slfrak_2(\F_p)$ in \cite{jezernik2025random}. Using different methods, we are able to handle any fixed Lie type.\footnote{Extremely rapid random generation also holds for finite simple groups of fixed Lie type, but the argument there is more conceptual: such groups are expanders with high probability \cite{breuillard2015expansion}.} Our approach involves constructing a characteristic $0$ Lie algebra that serves as a covering object for all Lie algebras $\gfrak$ over $\F_p$ of the same type. We show that this covering object contains a generating pair whose Lie operations exhibit exponential growth. Projecting this pair down to each $\gfrak$, we obtain elements that quickly generate a large portion of $\gfrak$. Applying the Schwartz-Zippel lemma, we deduce that the same holds for generic pairs. In the non-split case, several covering objects are required for number-theoretic reasons. We construct these independently to cover almost all primes using the Chebotarev density theorem.

\subsection{Reader's guide}

Rapid generation of nonabelian finite simple Lie algebras is established in \Cref{sec:diameters} in the two-step manner as described above. We recall the details of how classical finite Lie algebras are constructed in \Cref{sec:classical_lie_algebras}, and then prove extremely rapid random generation in \Cref{sec:diameters_classical}.

\subsection{Acknowledgements}

We thank Daniel Smertnig for discussions on number fields, Oliver Roche-Newton for pointing out state of the art sum-product results, and Daniele Dona and Sean Eberhard for helpful comments on an earlier version. We also thank the anonymous referees for their careful reading and suggestions.

\section{Diameters of finite simple Lie algebras}
\label{sec:diameters}

\subsection{Starting with a line}

Following \cite{dona-sum-bracket}, we define, for subsets $X, Y \subseteq \gfrak$, \emph{towers} of Lie brackets recursively as 
\[
    \tower_0(X) = \tower_0(X, Y) = \{0\}, \quad
    \tower_1(X) = \tower_1(X, Y) = X,
\]
and for $k \geq 2$,
\begin{align*}
    \tower_k(X) &= [X, \tower_{k - 1}] \cup [\tower_{k - 1}, X], \\
    \tower_k(X, Y) &= [Y, \tower_{k - 1}(X)] \cup [\tower_{k - 1}(X), Y] 
    \cup [X, \tower_{k - 1}(X, Y)] \cup [\tower_{k - 1}(X, Y), X].
\end{align*}
If $Y = \{y\}$ is a singleton, we write $\tower_k(X, y)$ instead of $\tower_k(X, Y)$. We also define
\[
    \tower_{\leq k}(X) = \bigcup_{j = 0}^k \tower_j(X) 
    \quad \text{and} \quad 
    \tower_{\leq k}(X, Y) = \bigcup_{j = 0}^k \tower_j(X, Y).
\]

Towers built over generating sets quickly produce spanning sets of Lie algebras. 

\begin{proposition}[\cite{dona-sum-bracket}, Proposition 3.3]
    \label[proposition]{dim_growth_dona}
    Let $A$ be a generating set of a Lie algebra $\gfrak$ over a field $K$ of dimension $d$. Then for any $k \in \N$ the set $\tower_{\leq k}(A)$ spans a vector subspace of $\gfrak$ of dimension at least $\min\{d, k\}$.
\end{proposition}

We now show that for any generating set $A$ and any nonzero element $b \in \gfrak$, the relative tower $\tower_{\leq d}(A, b)$ spans the entire Lie algebra. While \cite[Corollary 3.4]{dona-sum-bracket} established this result for $\tower_{\leq 2d}(A, b)$, it is noted there that the bound can be refined to $d$. Below, we provide this refinement, although it is not essential for our argument. For a subset $S$ of a vector space, we set $\lin S$ to denote the linear span of $S$.

\begin{lemma}
    \label[lemma]{lem:towering}
    Let $\liealg$ be a Lie algebra over a field $K$, let $A \subseteq \liealg$ be a subset of $\liealg$ and let $b \in \liealg$ be an element of $\liealg$. Then for any $m, n \in \N$ we have
    \[
        [\tower_m(A, b), \tower_n(A)] \subseteq \lin [A, \tower_{m + n - 1}(A, b)].
    \]
\end{lemma}
\begin{proof}
    Induction on $n$. For $n = 1$ the statement is trivial. Let $s \in \tower_m(A, b)$ and $t \in \tower_{n + 1}(A)$. Then up to a sign $t = [a, t']$ for some $a \in A$ and $t' \in \tower_n(A)$. The Jacobi identity gives
    \[
        [s, t] = [s, [a, t']] = [[s, a], t'] + [a, [s, t']].
    \]
    By the induction hypothesis we have $[[s, a], t'] \in \lin [A, \tower_{m + n}(A, b)]$ and $[s, t'] \in \lin [A, \tower_{m + n - 1}(A, b)] \subseteq \lin \tower_{m + n}(A, b)$. Thus we get 
    \[
        [s, t] \in \lin [A, \tower_{m + n}(A, b)]
    \]
    and the proof is complete.
\end{proof}

\begin{proposition}
    \label[proposition]{lem:dim_growth}
    Let $\mathfrak{g}$ be a simple Lie algebra over a field $K$ of dimension $d$. Let $A$ be a generating set of $\mathfrak{g}$ and let $b \in \liealg$ be a nonzero element. Then the set $\tower_{\leq k}(A, b)$ spans a vector subspace of $\liealg$ of dimension at least $k$ for any $k \leq d$. In particular, the set $\tower_{\leq d}(A, b)$ spans $\liealg$.
\end{proposition}

\begin{proof}
    Define $V_k = \lin(\tower_{\leq k}(A, b))$ for any $k \in \N$. Since $b \neq 0$, we have $\dim V_1 = 1$, so it is enough to show that $\dim V_{k + 1} \geq \min\{d, \dim V_k + 1\}$ for any $k \in \N$. 
    
    Assume that $\dim V_{k + 1} < \dim V_k + 1$. Then $V_{k + 1} = V_k$. Now take any element $v \in \tower_{\leq k + 2}(A,b)$. Up to a sign we either have $v = [a, t]$ for some $a \in A$ and $t \in \tower_{\leq k+1}(A,b)$, or $v = [b, t']$ for some $t' \in \tower_{\leq k+1}(A)$. In the former case we have $t \in V_{k + 1} = V_k$ and thus $v \in [a, V_k] \subseteq V_{k + 1} = V_k$. In the latter case we use Lemma~\ref{lem:towering} to get
    \[
        v \in \lin [A, \tower_{\leq k + 1}(A, b)] \subseteq \lin [A, V_{k + 1}] = \lin [A, V_k] \subseteq V_{k + 1} = V_k.
    \]
    Therefore $V_{k + 2} = V_k$ and thus $V_l = V_k$ for any $l \geq k$. It now follows by \Cref{dim_growth_dona} and \Cref{lem:towering} that 
    \begin{align*}
        [\liealg, V_k] &\subseteq \lin [\tower_{\leq d}(A), \tower_{\leq k}(A, b)] \subseteq \lin [A, \tower_{\leq d + k - 1}(A, b)]\\ 
        &\subseteq \lin \tower_{\leq d + k}(A, b) = V_{d + k} = V_k,
    \end{align*}
    which implies that $V_k$ is an ideal of $\liealg$. By simplicity of $\liealg$ we obtain $V_k = \liealg$, hence $\dim V_k = d$.
\end{proof}

We are now ready to prove the main result of this section, which is that of transporting a line covered by $A^k$ to independent directions and thus covering the whole Lie algebra.

\begin{proof}[Proof of \Cref{prop:full_line}]
    Suppose that $A^k$ contains a line spanned by a vector $v \in \gfrak$. By Proposition \ref{lem:dim_growth}, the Lie algebra $\liealg$ has a basis consisting of vectors $v_1, \ldots, v_d$, where $v_j \in \tower_{\leq j}(A, v) \subseteq A^j$. Take any element $u = \sum_{j = 1}^d \alpha_j v_j$ from $\liealg$. For any $j$ we can write
    \[
        v_j = [a_1, \ldots, a_l, v, a_{l + 1}, \ldots, a_{j'-1}] \quad \text{for some } a_i \in A, \ j' \leq j
    \]
    and since $\alpha_j v \in A^k$ we have $\alpha_j v_j = [a_1, \ldots, a_l, \alpha_j v, a_{l + 1}, \ldots, a_{j'-1}] \in A^{k + j - 1}$. It then follows that $u \in A^{kd + d(d - 1)/2}$.
\end{proof}

\subsection{Covering a line}

Our objective now is to show that we can achieve that $\ell(A^k)$ is large with a reasonably small $k$. The main input driving growth is the sum-product theorem for finite fields in the following form.

\begin{theorem}
    \label[theorem]{thm:sum_product}
    For every $\epsilon > 0$ there are $C, c > 0$ so that for every subset $X \subseteq \F_p$:
    \begin{enumerate}
        \item \cite{mohammadi2023attaining} If $C \leq \abs{X} \leq c p^{1/2}$, then $\max \{ \abs{X+X}, \abs{XX} \} \geq \abs{X}^{5/4 - \epsilon}$.
        \item \cite[Corollary 4]{roche-sum-product} If $\abs{X} > p^{4/7}$, then $\abs{XX + XX + XX} \geq c p$.
    \end{enumerate}
\end{theorem}

We first prove the product theorem in lines. The argument is similar to the one in \cite[Proof of Theorem 1.1, Step 2]{dona-sum-bracket}.

\begin{proof}[Proof of \Cref{line_growth}]
    Let $A$ be a generating set of $\gfrak$, let $\epsilon > 0$ and let $k \in \N$. Take $v \in \gfrak$ to be a nonzero element such that the maximum $\ell(A^k)$ is attained at the line spanned by $v$, \emph{i.e.}, the set $X = \{\alpha \in \F_p \mid \alpha v \in A^k\}$ has size $\abs{X} = \ell(A^k)$. Let $C, c > 0$ be the constants from the sum-product theorem above.

    The set $A^{2k}$ contains the set $(X+X)v$. Let $u \in \gfrak$ be any nonzero element such that $[v, u] \neq 0$ (such an element exists since $\gfrak$ is centerless). By Proposition~\ref{lem:dim_growth}, the element $u$ is a linear combination of elements from $\tower_{\leq \dim \gfrak}(A, v)$, so by bilinearity of the Lie bracket we can assume that $u \in \tower_{\leq \dim \gfrak}(A, v)$. Thus we can write 
    \[
        u = [a_1, \ldots, a_j, v, a_{j + 1}, \ldots, a_{d'-1}] \quad \text{for some } a_i \in A, \ d' \leq \dim \gfrak.
    \]
    and hence
    \[
        (XX)[v, u] = [Xv, [a_1, \ldots, a_j, Xv, a_{j + 1}, \ldots, a_{d'-1}]] \subseteq A^{2k + \dim \gfrak}.
    \]
    Therefore the set $A^{2k + \dim \gfrak}$ contains both $(X+X)v$ and $(XX)[v,u]$, hence
    \[
        \ell(A^{2k + \dim \gfrak}) \geq \max\{\abs{X+X}, \abs{XX}\}.
    \]
    
    We now invoke the sum-product theorem above. Since $\abs{X} \geq k$, taking $k > C$ ensures that $\abs{X} > C$. If $\abs{X} \leq c p^{1/2}$, then we are in the regime of the first part of the sum-product theorem, and we obtain $\ell(A^{2k + \dim \gfrak}) \geq \abs{X}^{5/4 - \epsilon}$, as desired. Suppose now that $\abs{X} > c p^{1/2}$. Applying the sum-product theorem to a subset of $X$ of size $\floor{c p^{1/2}}$, we get $\ell(A^{2k + \dim \gfrak}) \geq c^{5/4 - \epsilon} p^{5/8 - \epsilon} > p^{4/7}$ for sufficiently large $p$ (depending only on $\epsilon$). It now follows from the second part of the sum-product theorem that $\ell(A^{6(2k + \dim \gfrak)}) \geq cp$. The Cauchy-Davenport theorem \cite[Theorem 5.4]{tao2006additive} implies that $\ell(A^{C'(2k + \dim \gfrak)}) = p$ for some constant $C'$ depending only on $\epsilon$, and we are done.
\end{proof}

Let us now show how this phenomenon can be iterated to achieve a large enough $\ell(A^k)$ and obtain the stated diameter bound.

\begin{proof}[Proof of Theorem \ref{thm:babai_conjecture}]
    Let $\epsilon > 0$, let $d$ be the dimension of $\liealg$ over $\F_p$, and set $\alpha = \ceil{ \max \{ 1 / \epsilon, C, 3 \} }$, where $C$ is the constant appearing in \Cref{line_growth}. Note that for every integer $h \geq \alpha d$ we have $\floor{(2 + \epsilon)h} \geq 2h + d$. Thus, the product theorem in lines gives 
    \[
        \ell(A^{\floor{(2 + \epsilon)h}}) \geq \ell(A^{h})^{5/4 - \epsilon} \quad \text{or} \quad \ell(A^{C \floor{(2 + \epsilon)h}}) = p.
    \]
    We use the above $k$ times with $h = \alpha d$, $\lfloor (2+\epsilon) \alpha d \rfloor$, $\lfloor (2+\epsilon)^2  \alpha d \rfloor$, \emph{etc}. Supposing that after $k$ steps, we have $\ell(A^{C \lfloor (2 + \epsilon)^k \alpha d \rfloor}) < p$, then we must be in the regime of the first case each time, and hence, as $\ell(A^{\alpha d}) \geq \alpha d$, we obtain
    \[
        \ell(A^{\lfloor (2+\epsilon)^k \alpha d \rfloor}) \geq (\alpha d)^{(5/4 - \epsilon)^k}.
    \]
    Taking $k = \ceil{\log\log(p) / \log(5/4 - \epsilon)}$, the right hand side exceeds $p$. Therefore, we must have $\ell(A^{C\floor{(2 + \epsilon)^k \alpha d}}) = p$ for this value of $k$. \Cref{prop:full_line} now implies
    \[
        \diam(\gfrak, A) 
        \leq C \floor{(2 + \epsilon)^k \alpha d} d + d^2 
        = O_{\epsilon}\left(d^2 (\log p)^{\log(2 + \epsilon) / \log(5/4 - \epsilon)}\right).
    \]
    By taking $\epsilon$ small, the exponent approaches $\log 2 / \log(5/4)$.
\end{proof}

\begin{remark} \label{remark on stronger sum product theorem}
If a stronger sum-product theorem than \Cref{thm:sum_product} were available, the exponent $D$ in \Cref{thm:babai_conjecture} would improve accordingly. More precisely, suppose that for some $\alpha > 1$ the following held: for every $\epsilon > 0$ there are constants $C,c > 0$ such that every subset $X \subseteq \F_p$ with $C \leq \abs{X} \leq p^c$ satisfies $\max\{\abs{X+X},\abs{XX}\} \geq \abs{X}^{\alpha - \epsilon}$. Then the same argument as above would yield
\[
    \diam(\gfrak,A)
    =
    O_\epsilon\left((\dim \gfrak)^2(\log p)^{\log 2/\log(\alpha-\epsilon)}\right).
\]
In particular, any exponent $\alpha > \sqrt{2}$ would bring the logarithmic exponent below $2$. Since our proof relies on \Cref{prop:full_line}, it cannot produce a bound that is subquadratic in $\dim \gfrak$. Thus $D=2$ is the natural limit of this particular strategy. 

This can be compared with the much stronger sum-product conjecture of Erdős \cite{erdos1976some} that predicts near-quadratic growth ($\alpha = 2 - o(1)$) in the setting of the reals. Recent work of Bloom--Sawin--Schildkraut--Zhelezov \cite{bloom2026sum} shows that this strongest form is false, and also gives related counterexamples in finite fields. For the purposes of the present method, near-quadratic sum-product growth is not needed, but it does set a barrier for the method.

We nevertheless believe that the true diameter bound for simple Lie algebras should be stronger, and we are not aware of any counterexamples that even $D=1$ might suffice. Below are two examples that illustrate how such extremely rapid generation can be achieved in some natural situations.
\end{remark}

\begin{example}    
    First, consider $\slfrak_2(\F_p)$ with its standard generators $e$ and $f$. One can generate the entire one-dimensional subspace spanned by $e$ in $O(\log p)$ steps: the operations $\ad_{[e,f]}$ and adding $e$ act on this line as multiplication by $2$ and translation by $1$, respectively. This is much faster than the general bound coming from the product theorem in lines, and since the dimension is fixed, Proposition \ref{prop:full_line} then yields extremely rapid generation. We show in the following sections that this phenomenon is generic for classical Lie algebras of fixed type.
    
    As a second, more intricate example with unbounded dimension, consider the Witt algebra\footnote{The \emph{Witt algebra} $W(p)$ is a $p$-dimensional simple Lie algebra over $\F_p$, generated by $e_{-1},\dots,e_{p-2}$ with Lie bracket $[e_i, e_j] = (j-i) e_{i+j}$ when $-1 \leq i+j \leq p-2$, and $[e_i, e_j] = 0$ otherwise.} $W(p)$ with standard generators $e_{-1}$ and $e_2$. The same method fills each line spanned by $e_i$ for $i \le 2$ in $O(\log p)$ steps. Now, any element of the form $\sum_{j=2}^k \alpha_j e_j$ can be produced in $O(k\log p)$ steps: the base case $k=2$ is immediate, and the inductive step follows from 
    \[
        \textstyle \sum_{j=2}^{k+1} \alpha_j e_j
        = \alpha_2 e_2 + \left[e_1,\sum_{j=2}^k (j-1)^{-1}\alpha_{j+1} e_j\right].
    \]
    The diameter with respect to $e_{-1},e_2$ is thus $O(p\log p)$, which is $O(\log\abs{W(p)})$.
\end{example}

\begin{remark}
    The same arguments used to establish diameter bounds for nonabelian finite simple Lie algebras apply equally to finite-dimensional simple associative algebras over $\F_p$. By Wedderburn-Artin's theorem, such algebras are matrix algebras over finite extensions of $\F_p$, and these hence admit the same polylogarithmic diameter bounds with respect to any generating set.
\end{remark}

\section{Classical finite Lie algebras}
\label{sec:classical_lie_algebras}

In this section, we provide a detailed overview of the classical Lie algebras over finite fields, with particular emphasis on their forms over $\F_p$.

\subsection{Classical Lie algebras}

The finite-dimensional simple Lie algebras over the complex numbers $\C$ are well-known and they are classified into the following families:
\[
\slfrak_n(\C), \
\sofrak_{2n}(\C), \
\spfrak_{2n}(\C), \
\sofrak_{2n + 1}(\C), \
\gfrak_2(\C), \
\ffrak_4(\C), \
\efrak_6(\C), \
\efrak_7(\C), \
\efrak_8(\C).
\]
Each family corresponds to a root system $\Phi$, which can be represented combinatorially by a Dynkin diagram $D$.

Analogues of these Lie algebras exist over any field (including fields of positive characteristic), and there is a unified process for constructing them. We begin with the complex Lie algebra $\gfrak(\C)$ and its \textit{Chevalley basis} $\basis$ \cite{fulton2013representation}, which consists of
\[ \{ e_\alpha \mid \alpha \in \Phi \} \cup \{ h_\alpha \mid \alpha \in \Delta \}, \]
where $\Delta$ is a set of \emph{simple roots}, also referred to as a \emph{base} for $\Phi$. The key property of Chevalley bases is that all the structure constants of the Lie algebra $\gfrak(\C)$ with respect to this basis are integers. Consequently,  the lattice
\[\gfrak(\Z) = \langle \basis \rangle_\Z \subseteq \gfrak(\C)\]
spanned by the Chevalley basis also has all its structure constants with respect to $\basis$ which are integers. Moreover, this lattice is independent of the choice of $\Delta$ and is closed under the Lie bracket. Using it, we can define for any field $K$ the Lie algebra
\[
\gfrak(K) = \gfrak(\Z) \otimes_\Z K. 
\]
When the field $K$ is algebraically closed and of positive characteristic, all such algebras $\gfrak(K)$ are collectively referred to as \emph{classical Lie algebras}.

There are, however, more finite-dimensional simple Lie algebras over an algebraically closed field of positive characteristic than just the classical ones (factored by their centre). \emph{Filtered Lie algebras of Cartan type} are certain deformations of finite dimensional graded Lie algebras of Cartan type. The most well-known example of these is the Witt algebra $W(p)$ that has no finite dimensional analogue in characteristic $0$. In the low characteristic $p = 5$, there is another series of finite-dimensional simple Lie algebras called \emph{Melikian Lie algebras}. The classification theorem of Block-Wilson-Strade-Premet \cite{strade2004simple} states that these Lie algebras are the only finite-dimensional simple Lie algebras provided that $p > 3$. 

\subsection{Forms of classical Lie algebras over finite fields}

The finite-dimensional simple Lie algebras over $\F_p$ are not fully classified. However, we can describe the $\F_p$-Lie algebras that become a simple classical Lie algebra after extension of scalars to the algebraic closure $\overline{\F_p}$. For $p > 3$, these all arise, after quotienting by the center, from \emph{$\F_p$-forms of the classical Lie algebras} listed above, \emph{i.e.}, those $\F_p$-Lie algebras that become isomorphic to $\gfrak(\overline{\F_p})$ after extension of scalars. In this paper, we focus exclusively on these forms, and in particular we do not say anything about Lie algebras of Cartan type (see \Cref{remark:witt}).

A form of $\gfrak(\overline{\F_p})$ is said to be \emph{split} if it is isomorphic as an $\F_p$-algebra to $\gfrak(\F_p)$, constructed via the Chevalley basis as above. For some classical root systems, there also exist non-split forms, which arise from automorphisms of the corresponding Dynkin diagrams. Both types of forms can be constructed using a Galois descent technique starting from $\gfrak(\overline{\F_p})$, as explained in \cite[Chapter IV]{seligman2012modular}. 

For our purposes, we require an analogue of the Chevalley basis for the non-split $\F_p$-forms. As such a basis does not appear to be readily available in the literature, we follow the original construction of Steinberg \cite{steinberg1959variations} and provide all the necessary details here.

\subsubsection{Steinberg's construction}

Let $\gfrak(\C)$ be a classical Lie algebra, and let $D$ be its Dynkin diagram. Suppose that $D$ admits a symmetry, which we denote by $\vartheta$. The only possibilities in which $\vartheta$ is nontrivial are the following (see \Cref{fig:dynkin-labeling}).
\begin{itemize}
    \item The type is $A_n$ ($n \geq 2$), and $\vartheta$ is the horizontal flip of $D$ that swaps the nodes $i$ and $n + 1 - i$ for $1 \leq i \leq n$.
    \item The type is $D_n$ ($n \geq 4$), and $\vartheta$ is the vertical flip of $D$ that swaps the nodes $n - 1$ and $n$.
    \item The type is $D_4$, and $\vartheta$ is the automorphism of $D$ that cyclically permutes the nodes $1, 3, 4$ (or in reverse order) and fixes the node $2$. 
    \item The type is $E_6$, and $\vartheta$ is the vertical flip of $D$ that swaps the nodes $1$ and $6$, and the nodes $2$ and $5$, while fixing the nodes $3$ and $4$.
\end{itemize}
\begin{figure}[t]
\centering
\begin{tikzpicture}[
    scale=1.5,
    dot/.style={circle, fill=black, inner sep=1.5pt},
    edge/.style={thick},
    num/.style={font=\scriptsize, inner sep=1pt},
    symoneway/.style={->, >=stealth, bend left=45, dashed},
    symtwowaysright/.style={<->, >=stealth, bend right=45, dashed},
    symtwowaysleft/.style={<->, >=stealth, bend left=45, dashed}
]

\begin{scope}[xshift=0cm]
\node[dot,label={[num]above:$1$}] (A1) at (0,0) {};
\node[dot,label={[num]above:$2$}] (A2) at (0.45,0) {};
\node[dot,label={[num]above:$3$}] (A3) at (0.9,0) {};
\draw[edge] (A1)--(A2)--(A3);
\draw[symtwowaysright] (A1) to (A3);
\end{scope}

\begin{scope}[xshift=1.75cm]
\node[dot,label={[num]above:$1$}] (D51) at (0,0) {};
\node[dot,label={[num]above:$2$}] (D52) at (0.45,0) {};
\node[dot,label={[num]above left:$3$}] (D53) at (0.9,0) {};
\node[dot,label={[num]above:$4$}] (D54) at (1.125,0.4005) {};
\node[dot,label={[num]below:$5$}] (D55) at (1.125,-0.4005) {};
\draw[edge] (D51)--(D52)--(D53)--(D54);
\draw[edge] (D53)--(D55);
\draw[symtwowaysleft] (D54) to (D55);
\end{scope}

\begin{scope}[xshift=4.4cm]
\node[dot,label={[num]above left:$2$}] (D42) at (0,0) {};
\node[dot,label={[num]above left:$1$}] (D41) at (-0.45,0) {};
\node[dot,label={[num]above:$3$}] (D43) at (0.225,0.4005) {};
\node[dot,label={[num]below:$4$}] (D44) at (0.225,-0.4005) {};
\draw[edge] (D41)--(D42);
\draw[edge] (D42)--(D43);
\draw[edge] (D42)--(D44);
\draw[symoneway] (D41) to (D43);
\draw[symoneway] (D43) to (D44);
\draw[symoneway] (D44) to (D41);
\end{scope}

\begin{scope}[xshift=5.5cm]
\node[dot,label={[num]above:$1$}] (E1) at (0,0) {};
\node[dot,label={[num]above:$2$}] (E2) at (0.45,0) {};
\node[dot,label={[num]above left:$3$}] (E3) at (0.9,0) {};
\node[dot,label={[num]above:$4$}] (E4) at (0.9,0.45) {};
\node[dot,label={[num]above:$5$}] (E5) at (1.35,0) {};
\node[dot,label={[num]above:$6$}] (E6) at (1.8,0) {};
\draw[edge] (E1)--(E2)--(E3)--(E5)--(E6);
\draw[edge] (E3)--(E4);
\draw[symtwowaysright] (E1) to (E6);
\draw[symtwowaysright] (E2) to (E5);
\end{scope}

\end{tikzpicture}
\caption{Symmetries of Dynkin diagrams of types $A_3$, $D_5$, $D_4$ and $E_6$.}
\label{fig:dynkin-labeling}
\end{figure}
In all these cases, $\vartheta$ defines an automorphism of $\gfrak(\C)$ that preserves the lattice $\gfrak(\Z)$ with the property that, for every element of the Chevalley basis,
\[ \vartheta(h_\alpha) = h_{\vartheta\alpha} \quad \hbox{and} \quad \vartheta(e_\alpha) = \pm e_{\vartheta\alpha} ,\]
where the appropriate sign for $\vartheta(e_\alpha)$ is given in \cite[Lemma 3.2 and Section 10]{steinberg1959variations}. In particular, if $\alpha \in \Delta$ is simple, then $\vartheta(e_\alpha) =  e_{\vartheta\alpha}$. The automorphism $\vartheta$ thus induces an automorphism of the Lie ring $\gfrak(\Z)$.

\begin{example}
Let $\gfrak = \slfrak_3$. The basis of the root system consists of $\alpha$, $\beta$ with corresponding root vectors $E_{12}$, $E_{23}$, and the highest root vector $E_{13} = [E_{12}, E_{23}]$. The automorphism $\vartheta$ of the Dynkin diagram swaps the two basis elements. We have $\vartheta(E_{13}) = \vartheta([E_{12}, E_{23}]) = - E_{13}$. On the other hand, $\vartheta$ fixes the Cartan subalgebra setwise.
\end{example}

Now let $K$ be any field with an extension $E$ of degree $\abs{\vartheta}$, the order of $\vartheta$ (note that $\abs{\vartheta} = 1, 2, 3$). Let $\sigma$ be an element of $\Gal(E/K)$ of order $\abs{\vartheta}$. The composition $\Theta = \vartheta \sigma$ is a semilinear automorphism of the Lie algebra $\gfrak(E)$. The set of its fixed points $\gfrak(E)^\Theta$ is a Lie algebra over $K$ whose extension of scalars to $E$ is isomorphic to $\gfrak(E)$. By taking $K = \F_p$ for $p > 3$, the Lie algebras constructed in this way exhaust, after quotienting by the center, all isomorphism classes of non-split $\F_p$-forms (and, more generally, all non-split forms over any finite field). See, for instance, \cite[Chapter IV.6]{seligman2012modular}.

\subsubsection{Characteristic $0$ coverings}

Let $\gfrak(\F_{p^d})^\Theta$ be an $\F_p$-form of a classical Lie algebra, where $d = \abs{\vartheta}$. By construction, this Lie algebra has an associated characteristic $0$ Lie algebra $\gfrak(E)^\Theta$, where $E$ is a degree $d$ extension of $\Q$. We call this a \emph{$\Q$-covering Lie algebra} of $\gfrak(\F_{p^d})^\Theta$. It is itself a $\Q$-form of the Lie algebra $\gfrak(\C)$, and it contains the Lie ring $\gfrak(\integers_E)^\Theta$, referred to as the corresponding \emph{covering ring}, where $\integers_E$ is the ring of integers of the number field $E$. We have a natural surjection 
\[
\pi \colon \gfrak(\integers_E) \to \gfrak(\F_{p^d})
\]
by reducing to $\integers_E/p \cong \F_{p^d}$ as long as $p$ is inert in $E$. The map $\pi$ is equivariant with respect to the action of $\Theta$, so it induces a natural map $\pi^\Theta$ on the level of fixed points. 

\begin{lemma} \label[lemma]{lemma:covering_surjection}
Assume $p > d$. Then $\pi^\Theta \colon \gfrak(\integers_E)^\Theta \to \gfrak(\F_{p^d})^\Theta$ is a surjection.
\end{lemma}
\begin{proof}
Let $x \in \gfrak(\F_{p^d})^\Theta$. Let $\tilde x$ be a lift of $x$ to $\gfrak(\Z) \otimes \integers_E$. This element might not be fixed under $\Theta$, but we can average it over the orbit of $\Theta$ to obtain the element $y = \sum_{i = 1}^d \Theta^i(\tilde x)$. The element $y$ is fixed under $\Theta$, and it reduces to $dx \in \gfrak(\F_{p^d})$ modulo $p$. Since $p > d$, we thus obtain $x \in \image(\pi^\Theta)$.
\end{proof}

The averaging process over the $\langle \Theta \rangle$-orbits described in the proof above is a key step in constructing a basis for the $\Q$-covering Lie algebra $\gfrak(E)^\Theta$ that is analogous to the Chevalley basis in the split case. For each vector $x \in \gfrak(E)$, let $x^{\langle \Theta \rangle} = \sum_{t \in \langle \Theta \rangle} x^t$. 

\begin{proposition}
\label[proposition]{prop:steinberg_basis}
Let $E$ be a number field with basis $\mathcal{E}\subseteq \integers_E$ over $\Q$.
Let $\basis$ be the Chevalley basis of $\gfrak(E)$. Then 
\[
    \steinbergbasis = \left\{ (e b) ^{\langle \Theta \rangle} \mid b \in \basis, \ e \in \mathcal{E} \right\} \setminus \{ 0 \} \subseteq \gfrak(\integers_E)^\Theta
\]
is a spanning set of $\gfrak(E)^\Theta$ over $\Q$.
\end{proposition}
\begin{proof}
    Let $A$ be the endomorphism of $\gfrak(E)$ over $\Q$ mapping $x$ to its $\langle\Theta\rangle$-average $x^{\langle \Theta \rangle}/\abs{\langle \Theta \rangle}$. Note that $A$ maps to $\gfrak(E)^\Theta$ and preserves it, hence it is a projection onto $\gfrak(E)^\Theta$. Therefore a spanning set for $\gfrak(E)^\Theta$ can be obtained by mapping the basis $\{ eb \mid b \in \basis, e \in \mathcal{E} \}$ of $\gfrak(E)$ under $A$.
\end{proof}

\begin{example}
Let $\gfrak = \slfrak_3$ and $E = \Q(\omega)$ of degree $2$ with primitive element $\omega$ and equipped with an automorphism $\sigma$ of order $2$. A Chevalley basis for $\gfrak(E)$ is given by
\[
    \basis = \{ E_{12}, E_{23}, E_{13} \} \cup \{ E_{21}, E_{32}, E_{31} \} \cup \{ E_{11} - E_{22}, E_{22} - E_{33} \}.
\]
For each vector $x \in \slfrak_3(E)$, we compute its $\langle \Theta \rangle$-orbit as $x^{\langle \Theta \rangle} = x + x^\Theta$. For example, $(\omega E_{12})^{\langle \Theta \rangle} = \omega E_{12} + \omega^\sigma E_{23}$. Performing this on all the Chevalley basis elements and their $\omega$-multiples, we obtain
\[
    \steinbergbasis = \bigcup_{x \in E_{12}, E_{23}} \{ x^{\langle \Theta \rangle}, (\omega x)^{\langle \Theta \rangle} \} \cup 
    \{ (\omega E_{13})^{\langle \Theta \rangle}, (\omega E_{31})^{\langle \Theta \rangle} \} \cup
    \{ E_{11} - E_{22}, E_{22} - E_{33} \},
\]
a basis for the $\Q$-Lie algebra $\gfrak(E)^\Theta$. Its elements all belong to $\gfrak(\integers_E)^\Theta$.
\end{example}

We will use the previous proposition in the following way. Let $\steinbergbasis$ be a $\Z$-basis of the free abelian group $\gfrak(\integers_E)^\Theta$. Assume $p > d$ is inert in $E$. The map $\pi^\Theta$ is surjective, so $\pi(\steinbergbasis)$ spans $\gfrak(\F_{p^d})^\Theta$ over $\F_p$. Moreover, 
\[
|\steinbergbasis|
=
\dim_{\Q} \gfrak(E)^\Theta
=
\dim_{\F_p} \gfrak(\F_{p^d})^\Theta,
\]
since these are forms of the same classical Lie algebra. Thus $\pi(\steinbergbasis)$ is a spanning set of cardinality equal to the dimension, and hence a basis. 

\subsubsection{A favorable covering pair of elements}

In order to bound diameters of the $\F_p$-forms with respect to a random pairs of elements, we use a pair of elements in the covering Lie ring with favorable properties. The first such property is rapid growth. Say a Lie ring $\gfrak$ exhibits \emph{exponential growth} with respect to a finite (not necessarily generating) subset $S$ if there exists a constant $\gamma > 1$ such that for any $m \in \N$, we have $\abs{S^m} \geq \gamma^m$.

\begin{example}
The Lie ring $\slfrak_2(\Z)$ exhibits exponential growth with respect to the standard generating set $\{ e, f, h\}$. To see this, let $\phi \colon \slfrak_2(\Z) \to \Z$ be the map $\phi(X) = X_{12}$. We claim that for any $m \in \N$ and $x \in \phi(S^m)$, we have $2x, 2x + 1 \in \phi(S^{m + 2})$. Indeed, for $X \in S^m$ with $\phi(X) = x$, we have $2x = \phi([h, X]) \in \phi(S^{m + 1}) \subseteq \phi(S^{m + 2})$ and $2x + 1 = \phi([h,X] + e) \in \phi(S^{m + 2})$. Starting with $0, 1 \in \phi(S^1)$, we can generate all integers from $0$ to $2^m - 1$ in $\phi(S^{2m-1})$ with the above procedure. Thus, we have $\abs{S^{m}} \geq 2^{m/2}$ for any $m \in \N$.
\end{example}
An \emph{$\slfrak_2$-triple} in a Lie algebra $\gfrak$ over $K$ is a triple of elements $\{e, h, f\}$ satisfying $[h,e] = 2e$, $[h,f] = -2f$, and $[e,f] = h$. In other words, this is an embedded copy of the Lie algebra $\slfrak_2(K)$ inside $\gfrak$. If $K = \Q$, then the existence of an $\slfrak_2$-triple in $\gfrak$ implies that $\gfrak$ exhibits exponential growth with respect to the set $\{e, h, f\}$, and therefore with respect to \emph{any} generating set as shown in the following lemma.

\begin{lemma}
\label[lemma]{lemma:sl_2_triple_exp_growth}
If a Lie algebra $\gfrak$ over $\Q$ contains an $\slfrak_2$-triple, then it exhibits exponential growth with respect to any generating set.
\end{lemma}
\begin{proof}
Let $S$ be a generating set of $\gfrak$ and let $\{e, h, f\}$ be an $\slfrak_2$-triple in $\gfrak$. Let us first show that there are positive integers $a, b, m \in \N$ such that $ae, bh \in S^m$. Since $S$ generates $\gfrak$ as a Lie algebra over $\Q$, we may write (after clearing denominators)
\[
a_0 e = \sum_i n_i p_i,
\qquad
c_0 f = \sum_j m_j q_j,
\]
where $a_0, c_0 \in \N$, $n_i, m_j \in \Z$, and $p_i, q_j$ are Lie monomials in $S$. Bracketing the two expressions gives
\[
a_0 c_0 h 
= \sum_{i, j} n_i m_j [p_i, q_j].
\]
In case a coefficient $n_i m_j$ is negative, we can replace the corresponding Lie monomial $[p_i, q_j]$ by $-[p_i, q_j] = [q_j, p_i]$. Taking $b = a_0 c_0$, we thus obtain $bh \in S^{m'}$ for some $m' \in \N$. Now
\[
2 a_0 b e
= [bh, a_0 e]
= \sum_{i, j} n_i [bh, p_i],
\]
so again after possibly reversing the order of some Lie monomials, we obtain $a e \in S^{m}$ for some $a \in \N$ and $m \in \N$ with $m \geq m'$.

Let us now show how this implies exponential growth with respect to $S$. For any $k \in \N_0$, let $X_k = \{ t \in \Z \mid te \in S^k \}$. For $t \in X_k$, we have $2bt e = [bh, te] \in S^{k + m}$, so $2bt \in X_{k + m}$. Hence $2b X_k \subseteq X_{k + m}$, and iterating this $l$ times gives $(2b)^l X_k \subseteq X_{k + lm}$. Choose $l$ such that $(2b)^l > a$. Then the sets $(2b)^l X_k$ and $a + (2b)^l X_k$ are disjoint subsets of $X_{k + (l + 1)m}$. Therefore $\abs{X_{k + (l + 1)m}} \geq 2\abs{X_k}$ for every $k$. Starting with $\abs{X_0} = 1$, we obtain $\abs{X_{n(l + 1)m}} \geq 2^n$ for any $n \in \N$, and thus $\abs{S^{n(l + 1)m}} \geq 2^n$ for any $n \in \N$.
\end{proof}

We now show that our $\Q$-coverings always contain an $\slfrak_2$-triple.

\begin{proposition}
\label[proposition]{sl_2_triple_in_covering}
Every $\Q$-covering of an $\F_p$-form of a classical Lie algebra contains an $\slfrak_2$-triple.
\end{proposition}
\begin{proof}
Let $e_\lambda$ be a highest weight vector in $\gfrak(\Z) \subseteq \gfrak(E)$. The defining property of the highest weight vector is that it is a root vector that commutes with all the simple root vectors. Since $\vartheta$ is a Lie ring automorphism of $\gfrak(\Z)$ that permutes the simple roots, it follows that $\vartheta(e_\lambda)$ is also a highest weight vector. Hence $\vartheta(e_\lambda) = \pm e_\lambda$ for some sign. If the sign is positive, then $e_\lambda \in \gfrak(E)^\Theta$, and so $\{ e_\lambda, e_{-\lambda}, h_{\lambda} \}$ forms an $\slfrak_2$-triple in $\gfrak(E)^\Theta$. Note that $\vartheta$ is of order dividing $2$ or $3$, hence the negative sign can only occur if $\vartheta$ is of order $2$. In this case, if $E = \Q[\omega]/(f(\omega))$ with $f$ of degree $2$, we have $(\omega e_\lambda)^{\langle \Theta \rangle} = (\omega - \omega^\sigma) e_\lambda \in \gfrak(E)^\Theta$. Note that
\[
    [(\omega - \omega^\sigma) e_\lambda, (\omega - \omega^\sigma) e_{-\lambda}] = (\omega - \omega^\sigma)^2 h_\lambda,
\]
where $(\omega - \omega^\sigma)^2 = \Tr(\omega)^2 - 4 \fieldnorm{\omega} = \Delta(f) \in \Q$, the discriminant of $f$. Hence $\{ (\omega e_\lambda)^{\langle \Theta \rangle}, (\omega e_{-\lambda})^{\langle \Theta \rangle} / \Delta(f), h_\lambda \}$ forms an $\slfrak_2$-triple in $\gfrak(E)^\Theta$ that belongs to $\gfrak(\integers_E[1/\Delta(f)])$.
\end{proof}

This implies that a $\Q$-covering of an $\F_p$-form of a classical Lie algebra exhibits exponential growth with respect to any generating set $S$. The following shows that we can always exhibit a particularly favorable generating set and even pass it on to the covering ring.

\begin{theorem} \label[theorem]{generating_pair_in_covering_ring}
Let $E$ be a number field of degree $d$ with $d \leq 3$. The ring $\gfrak(\integers_E)^\Theta$ contains a pair of elements $x, y$ such that the following hold.
\begin{enumerate}
    \item The ring $\gfrak(\integers_E)^\Theta$ exhibits exponential growth with respect to $\{x, y\}$.
    \item The image of $\{ x,y \}$ modulo $p$ generates $\gfrak(\F_{p^d})^\Theta$ for all large enough $p$ (depending only on the Lie type $\gfrak$) that are inert in $E$.
\end{enumerate}
\end{theorem}
\begin{proof}
Semisimple Lie algebras over fields of characteristic $0$ can be generated by two elements \cite[Theorem 6]{kuranishi1951everywhere}. Hence, any $\Q$-covering $\gfrak(E)^\Theta$ contains a generating pair of elements $x,y$. After possibly multiplying $x,y$ by suitable integers, we may assume that they are integral, thus $x, y \in \gfrak(\integers_E)^\Theta$. In particular, every element of the basis $b \in \steinbergbasis$ can be expressed as a $\Q$-linear combination of elements from Lie balls centered at $x,y$. After clearing denominators, we obtain an integral linear combination of elements from Lie balls centered at $x,y$ that equals an integer multiple of $b$. Therefore the reductions modulo $p$ of $x,y$ generate $\gfrak(\F_{p^d})^\Theta$ as long as $p$ is large enough.
\end{proof}

\begin{remark} 
\label[remark]{remark:witt}
Our proofs below crucially depend on the existence of a covering ring with the properties described in the proposition above. This approach does not extend to Lie algebras of Cartan type, as these do not possess finite-dimensional analogues in characteristic $0$. For instance, unlike the classical Lie algebras, the Witt algebra $W(p)$ has $\dim \gfrak = p$, so there is no meaningful notion of ``bounded rank'' in this context. We do not know whether this algebra exhibits extremely rapid random generation as $p$ tends to infinity.
\end{remark}

\subsubsection{An extremal basis}

A non-zero element $x \in \gfrak$ is called \emph{extremal} if $\image (\ad_x)^2 \subseteq \langle x \rangle$. If, in addition, $(\ad_x)^2 = 0$, then $x$ is called a \emph{sandwich}. For example, in $\slfrak_n(\C)$, any $E_{ij}$ with $i \neq j$ is extremal. These elements play a crucial role in proving the sum-bracket theorem \cite[Section 5.1]{dona-sum-bracket}, and will also be important in our arguments below. For our purposes, we will need the following variant of \cite[Theorem 5.2]{dona-sum-bracket} for non-split forms of classical Lie algebras. For any $a,b \in \gfrak$, let $q_{a,b}$ be the quadratic map
\[
    q_{a,b} \colon \gfrak \to \gfrak, \quad z \mapsto [\ad_a(z), \ad_b(z)].
\]

\begin{proposition}
\label[proposition]{prop:extremal_basis}
Let $p > 5$, and let $\gfrak$ be an $\F_p$-form of a classical Lie algebra.
\begin{enumerate}
    \item There exists a basis $\mathcal{E}$ of $\gfrak$ consisting of extremal non-sandwich elements.
    \item There exists an element $a \in \mathcal{E}$ such that $q_{a,b} \neq 0$ for every $b \in \mathcal{E} \setminus \{ a \}$.
\end{enumerate}
\end{proposition}

\begin{proof}
As in the proof of \Cref{sl_2_triple_in_covering}, we can find an $\slfrak_2$-triple in the ambient Lie algebra $\gfrak(\F_{p^d})$ constructed from a scalar multiple of the highest weight vector $e_\lambda$. Note that $\ad_{e_\lambda}^2$ maps all root subspaces either to $0$ or to the highest weight subspace, so any scalar multiple of $e_\lambda$ is extremal. Since it is a part of an $\slfrak_2$-triple, it is not a sandwich. It now follows from \cite[Theorem 1.1]{cohen2008simple} that $\gfrak(\F_{p^d})^\Theta$ is generated by extremal elements. Therefore, it is also spanned by extremal elements and contains no sandwiches, by \cite[Lemma 2.4, Corollary 4.5]{cohen2001lie}. This proves the first part. For the second part, we can assume the spanning set $\mathcal{E}$ contains the suitable scalar multiples of the highest and lowest weight vectors. Let $a \in \mathcal{E}$ be the corresponding multiple of the lowest weight vector $e_{-\lambda}$. To conclude, we do explicit calculations, which are both extensive and Lie type specific. Hence, we defer the argument to Appendix~\ref{sec:extremal_bases}.
\end{proof}

The above observation is the only step in the proof of the sum-bracket theorem \cite[Theorem 1.1]{dona-sum-bracket} that depends on the explicit structure of the Lie algebras, namely the split $\F_p$-forms. By using the same argument together with \Cref{prop:extremal_basis}, the sum-bracket theorem therefore extends to non-split $\F_p$-forms of classical Lie algebras as well. In the next section, we will use the following version.

\begin{corollary}[\emph{cf.} Theorem 1.1 in \cite{dona-sum-bracket}]
\label[corollary]{corollary:sum_bracket}
There is an absolute constant $c > 0$ such that the following holds. 
Let $\gfrak$ be an $\F_p$-form of a classical Lie algebra and $A$ any generating set of $\gfrak$. Then
\[
\abs{A^k} \geq \min \{ \abs{A}^{1+c}, \abs{\gfrak} \},
\]
where $k$ is a constant depending only on the Lie type of $\gfrak$.
\end{corollary}

\section{Diameters of classical finite Lie algebras}
\label{sec:diameters_classical}

We are now ready to prove that all the $\F_p$-forms of classical Lie algebras exhibit extremely rapid random generation as $p$ tends to infinity. Here is a brief outline of the argument. 

Let us first consider the split forms, since the technicalities of the non-split case can obscure the main ideas. For split forms, we do not need to worry about automorphisms of the Dynkin diagram. The covering map from \Cref{lemma:covering_surjection} is simply
\[
    \gfrak(\Z) \to \gfrak(\F_p).
\]
Our argument heavily utilizes this map in order to obtain uniform results over all large enough primes. We first show that the coefficients of the balls in $\gfrak(\Z)$ with respect to any pair of elements can grow at most exponentially in the radius of the ball. This means we can do up to about $\log p$ steps and still ensure that the coefficients will not loop around the field $\F_p$, so the ball of radius about $\log p$ will be mapped injectively by the quotient projection to $\gfrak(\F_p)$. Since balls with respect to a favorable pair of elements grow exponentially in the radius (as per the previous section), we can generate, after about $\log p$ steps, at least $p^\delta$ elements in $\gfrak(\F_p)$ for some $\delta > 0$. It follows from the Schwartz-Zippel lemma that the same conclusion then holds for generic pairs of elements. We finish off by using Dona's theorem to show that we generate the whole Lie algebra $\gfrak(\F_p)$ in $O(\log p)$ steps.

A similar, but more involved argument works for the non-split forms. The main difference is that the covering map is now $\gfrak(\integers_E)^\Theta \to \gfrak(\F_{p^d})^\Theta$. As $p$ varies, the characteristic $0$ object depends on the realization of $\F_{p^d}$, and we cannot cover all the cases with a single lattice in $\gfrak(\C)$. 

For this reason, we exhibit several covering objects, each of which covers a large proportion of the primes $p$, and collect them all together using tools from algebraic number theory to cover a set of primes of density arbitrary close to $1$.

\begin{remark}

Our proof can be compared to the argument for the group $\SL_2(\F_p)$ in \cite[Lemma 6.4 and Corollary 6.5]{helfgott2008growth}, but there is an important distinction. In the group case, most pairs of elements do not satisfy any identity of length $\leq c \log p$ (essentially because $\SL_2(\Z)$ contains a free subgroup). However, classical Lie algebras can satisfy polynomial identities whose degree depends only on the Lie type. For example, the Lie algebra $\slfrak_2(\F_p)$ satisfies the following polynomial identity (see \cite[Theorem 4]{drensky2021weak}):
\[
    [[[x_2, x_3], [x_4, x_1]], x_1] + [[[x_2, x_1], [x_3, x_1]], x_4] = 0.
\]
Replacing each $x_i$ with $[X, Y, \dots, Y]$ (with $i-1$ occurrences of $Y$) yields a nontrivial polynomial identity in two variables $X, Y$ on $\slfrak_2(\F_p)$ of degree $11$. (For instance, it is not an identity in $\slfrak_3(\F_p)$.) This shows that not all elements in $\slfrak_2(\Z)$ of length $\leq c \log p$ with respect to any generating set are distinct. Nevertheless, we prove that by projecting small balls from $\slfrak_2(\Z)$, we obtain exponentially many distinct elements.

\end{remark}

\subsection{Exponential growth and its projection}

\subsubsection{Growth of coefficients is at most exponential}

Let $\steinbergbasis$ be a basis of the free abelian group $\gfrak(\integers_E)^\Theta$.\footnote{Note that $\steinbergbasis$ is not necessarily the same as the spanning set constructed in the previous section. The latter generates an abelian subgroup of finite index in $\gfrak(\integers_E)^\Theta$.} For an element $x \in \gfrak(\integers_E)^\Theta$, expressed as $x = \sum_{b \in \steinbergbasis} x_b b$ with $x_b \in \Z$, let
\[
    \norm{x} = \max \{ \abs{x_b} \mid b \in \steinbergbasis \}.
    \]
For a finite subset $S$ of $\gfrak(\integers_E)^\Theta$, write $\norm{S} = \max \{ \norm{x} \mid x \in S \}$. In particular, let
\[
N = \norm{[\steinbergbasis, \steinbergbasis]} = \max \left\{ \norm{[b, b']} \mid b, b' \in \steinbergbasis \right\}.
\]
We have $N \geq 1$, and using it we can bound the growth of the coefficients of the covering ring with respect to the basis $\steinbergbasis$ as follows.

\begin{lemma}
    \label[lemma]{lemma:entry_bound}
    For any $S \subseteq \gfrak(\integers_E)^\Theta$ and any $m \in \N$, we have
    \[
        \norm{S^m} \leq (\abs{\steinbergbasis}^2 N)^{m - 1}\norm{S}^m.
    \]
\end{lemma}
\begin{proof}
    Induction on $m$. The base case $m = 1$ is trivial. Assume now that $m > 1$. For any nonzero $x \in S^m$, there are $y \in S^i$ and $z \in S^j$ such that $i + j = m$ and either $x = y + z$ or $x = [y, z]$. In the first case, we have $\norm{x} \leq \norm{y} + \norm{z}$, which is, by induction, clearly at most the claimed upper bound. In the second case, we have, by induction,
    \[
        \norm{x} 
        = \max_{b \in \steinbergbasis} \abs{[y,z]_b} 
        \leq \abs{\steinbergbasis}^2 N \norm{y} \norm{z}
        = (\abs{\steinbergbasis}^2 N)^{m - 1}\norm{S}^m.
    \]
    This completes the proof.
\end{proof}

\subsubsection{Initial growth of projected elements}

We will show that the projection of the favorable pair of elements in the covering ring exhibits initial growth in $\gfrak(\F_{p^d})^\Theta$, meaning that it covers a large proportion of the Lie algebra after a small number of steps.

\begin{lemma}
    \label[lemma]{lemma:initial_growth}
    Let $E$ be a number field of degree $d$ with $d \leq 3$. Let $x,y \in \gfrak(\integers_E)^\Theta$ be the elements provided by Proposition \ref{generating_pair_in_covering_ring}. Let $S = \{x, y\}$. 
    There are constants $c, \delta > 0$ such that 
    \[
        \abs{\pi(S)^{\floor{c\log p}}} \geq p^\delta
    \]
    for all large enough primes $p$ that are inert in $E$.
\end{lemma}
\begin{proof}
    Let $\gamma$ be the exponential growth constant of $\gfrak(\integers_E)^\Theta$ with respect to the set $S$. Take
    \[
        \delta = \frac{2}{3} \frac{\log \gamma}{\log (\abs{\steinbergbasis}^2 N \norm{S})}, \quad
        c = \frac{3 \delta}{2 \log \gamma}, \quad
        m = \floor{c\log p}, \quad
        p \geq \gamma^{2/\delta},
    \]
    as well as $p$ large enough so that \Cref{generating_pair_in_covering_ring} holds. With these choices, we have
    \[
        \abs{S^m} 
        \geq \gamma^m
        \geq \gamma^{-1} p^{c \log \gamma}
        = \gamma^{-1} p^{3 \delta/2}
        \geq p^{\delta}.
    \]
    Using \Cref{lemma:entry_bound} and taking into account that $c = 1/\log (\abs{\steinbergbasis}^2 N \norm{S})$, we also have
    \[
        \norm{S^m} 
        \leq \frac{1}{\abs{\steinbergbasis}^2 N} \left(\abs{\steinbergbasis}^2 N \norm{S}\right)^m 
        \leq \frac{1}{2} p^{c \log (\abs{\steinbergbasis}^2 N \norm{S})}
        = \frac{1}{2} p.
    \]
    The projection $\pi$ is thus injective on $S^m$ and so we have
    \[
        \abs{\pi(S)^m} = \abs{S^m} \geq p^\delta,
    \]
    completing the proof.
\end{proof}

\subsubsection{Initial growth of random pairs}

We now show that the same conclusion as in the previous lemma holds with high probability for random pairs of elements of $\gfrak(\F_{p^d})^\Theta$ as $p$ tends to infinity.

Let us first introduce some notation. A \emph{Lie word in two letters} is an element $w$ of the free Lie ring on two generators $x_1, x_2$. The \emph{length} of the Lie word $w$ is the smallest number $k$ such that $w \in \{x_1, x_2\}^k$. For any Lie algebra $\gfrak$ over $K$ and two elements $a,b \in \gfrak$, there is a unique homomorphism from the free Lie algebra $F = \langle x_1, x_2 \rangle$ to $\gfrak$ sending $x_1$ to $a$ and $x_2$ to $b$. This gives a well-defined \emph{word map} $w \colon \gfrak \times \gfrak \to \gfrak$ by evaluating the Lie word $w$ at any two elements $a, b$. 

\begin{lemma}
    \label[lemma]{random_pairs_generate_a_large_subset}
    For every number field $E$ of degree $d$, there are constants $c, \delta > 0$ such that the following holds for all large enough primes $p$ that are inert in $E$. Let $X,Y \in \gfrak(\F_{p^d})^\Theta$ be independent uniformly random elements. Then
    \[
        \Pr_{X, Y}\left(\abs{\{X, Y\}^{\floor{c\log p}}} < p^\delta\right) \leq \frac{c\log p}{p^{1 - 2\delta}},
    \]
\end{lemma}
\begin{proof}
    Let $p$ be inert in $E$. Thus $\gfrak(E)^\Theta$ is a $\Q$-covering of $\gfrak(\F_{p^d})^\Theta$. Let the values of $c,\delta$ and the elements $x,y$ be as in the previous lemma. Let $m = \floor{c\log p}$. Let $T \subseteq \{x, y\}^m \subseteq \gfrak(\F_{p^d})^\Theta$ be a subset of size $\floor{p^\delta}$. To each element $u \in T$ we assign a Lie word $w_u$ in two letters of length at most $m$ such that $w_u(x, y) = u$. Let 
    \[
        I = \{ (u,v,X,Y) \in T \times T \times \gfrak(\F_{p^d})^\Theta \times \gfrak(\F_{p^d})^\Theta \mid u \neq v \ \land \ w_u(X, Y) = w_v(X, Y) \},
    \]
    equipped with the natural projections 
    \[
    p_T \colon I \to T \times T, \quad p_{\gfrak} \colon I \to \gfrak(\F_{p^d})^\Theta \times \gfrak(\F_{p^d})^\Theta. 
    \]
    We claim that for distinct elements $u, v \in T$, the preimage $p_T^{-1}(u, v)$ is of smallish size. Indeed, since $u \neq v$, at least one of the coefficients of the element $u-v$ in the basis $\pi(\steinbergbasis)$ is nonzero. Now, for any $X = \sum_{b \in \steinbergbasis} X_b \pi(b)$ and $Y = \sum_{b \in \steinbergbasis} Y_b \pi(b)$ in $\gfrak(\F_{p^d})^\Theta$, the corresponding coefficient of the element $w_u(X, Y) - w_v(X, Y)$ in the basis $\pi(\steinbergbasis)$ is a polynomial $P$ in the $2 \abs{\steinbergbasis}$ variables $X_b, Y_b$ of total degree at most $m$. It follows from the Schwartz-Zippel lemma that we then have $w_u(X, Y) = w_v(X, Y)$ for at most $m p^{2 \abs{\steinbergbasis} - 1}$ pairs $X, Y$. In other words, the fibers of $p_T$ over distinct pairs are of size at most $m p^{2 \abs{\steinbergbasis} - 1}$. Since $T$ is of size $\floor{p^\delta}$, we conclude that 
    \[
        \abs{I} \leq m p^{2 \abs{\steinbergbasis} - 1 + 2 \delta}
        \leq c p^{2 \abs{\steinbergbasis} - 1 + 2\delta} \log p
        = \frac{c \log p}{p^{1 - 2 \delta}} \abs{\gfrak(\F_{p^d})^\Theta}^2.
    \]
    Therefore the projection $p_{\gfrak}(I)$ is bounded by the same value. But this projection consists of precisely of those pairs $(X, Y)$ for which the elements $\{w_u(X, Y) \mid u \in T\}$ are not pairwise distinct. The lemma follows.
\end{proof}

\subsubsection{Generation by random pairs}

Next, we prove that two random elements $X, Y$ generate the entire Lie algebra with high probability as $p$ tends to infinity. As in the previous argument, this follows from the Schwartz-Zippel lemma and the fact that at least one generating pair exists when the underlying Lie type is fixed.\footnote{For varying Lie types and low characteristics, two-generation does not directly follow from the characteristic zero case. See \cite{cantor2025two} for the example of $\slfrak_n(\F_p)$ as $n \to \infty$.}

\begin{proposition}
Let $\gfrak$ be an $\F_p$-form of a classical Lie algebra. Let $X, Y \in \gfrak$ be independent uniformly random elements. Then $X,Y$ generate the Lie algebra $\gfrak$ with probability tending to $1$ as $p \to \infty$ and the Lie type of $\gfrak$ is fixed.
\end{proposition}

\begin{proof}
Let $\gfrak = \gfrak(\F_{p^d})^\Theta$, and let $E$ be a number field of degree $d$ in which $p$ is inert. By \Cref{generating_pair_in_covering_ring}, $\gfrak$ is generated by the elements $\pi(x), \pi(y)$ for sufficiently large $p$, depending on the Lie type of $\gfrak$. Following \cite[proof of Proposition 1.1.3]{bois2009generators}, we can find $\dim(\gfrak)$ Lie monomials that span $\gfrak$. These monomials are of the form
\begin{equation}
    [\pi(y),[\pi(y), \dots ,[\pi(y),[\pi(x),[\pi(y), \dots ,[\pi(y),[\pi(x),[\pi(y), \dots [\pi(y), \pi(x)] \dots]]
\end{equation}
with at most $\dim(\gfrak)$ occurrences of $\pi(x)$, and at most $\dim(\gfrak)-1$ of $\pi(y)$ between any two occurrences of consecutive $\pi(x)$ (also before the first $\pi(x)$).

For any $X = \sum_{b \in \steinbergbasis} X_b \pi(b)$ and $Y = \sum_{b \in \steinbergbasis} Y_b \pi(b)$ in $\gfrak$, where $X_b, Y_b$ are $2 \dim(\gfrak)$ variables in $\F_p$, each monomial as above evaluates to an element of $\gfrak$, which can be expressed in terms of the basis $\pi(\steinbergbasis)$. The total degree of any coefficient appearing in the expansion of these monomials in terms of the variables $X_b, Y_b$ is at most $\dim(\gfrak)^2$. When $X = \pi(x)$ and $Y = \pi(y)$, a subset of these monomials forms a basis of $\gfrak$, so the determinant of the matrix formed by the coefficients of these monomials is a nonzero polynomial in the variables $X_b, Y_b$ with total degree at most $\dim(\gfrak)^3$. By the Schwartz-Zippel lemma, the probability that this determinant vanishes for random pairs $X, Y$ is at most $\dim(\gfrak)^3/p$. Therefore, the proportion of pairs $X, Y$ for which the corresponding $\dim(\gfrak)$ monomials are linearly independent approaches $1$ as $p \to \infty$.
\end{proof}

\subsubsection{Completing the covering}

Once the two random elements $X, Y$ generate a subset of size at least $p^\delta$, we can complete the covering of the whole Lie algebra in a modest number of extra steps using the sum-bracket theorem. 

\begin{lemma}
    \label[lemma]{completing_the_covering_from_a_large_subset}
    Let $\gfrak$ be an $\F_p$-form of a classical Lie algebra. For every $\delta > 0$ and every generating subset $A \subseteq \gfrak$ with size $\abs{A} \geq \abs{\gfrak}^\delta$, we have $A^m = \gfrak$ for some constant $m$ depending only on $\delta$ and the Lie type of $\gfrak$. 
\end{lemma}

\begin{proof}
    By \Cref{corollary:sum_bracket}, there exist an absolute constant $c > 0$ and a constant $k$ (depending only on the Lie type of $\gfrak$) such that for all $l \in \N$, we have
    \[
        \abs{A^{k^l}} \geq \min\left\{\abs{A}^{(1 + c)^l}, \abs{\gfrak}\right\}.
    \]
    Taking $l \geq \log(1 / \delta) / \log(1 + c)$, we have $\abs{A}^{(1 + c)^l} \geq \abs{\gfrak}$, and thus $A^{k^l} = \gfrak$.
\end{proof}

\subsection{Diameters of split forms}

\begin{theorem} \label[theorem]{theorem:diameter_bound_split}
    Let $\gfrak(\F_p)$ be a split form of a classical Lie algebra. Then there is a constant $C > 0$, depending only on the Lie type $\gfrak$, such that for a uniformly random pair of elements $X, Y \in \gfrak(\F_p)$, we have
    \[
        \diam ( \gfrak(\F_p), \{X, Y\} ) \leq C \log p
    \]
    with probability tending to $1$ as $p \to \infty$.
\end{theorem}
\begin{proof}
The statement immediately follows by combining \Cref{random_pairs_generate_a_large_subset} (with $E = \Q$) and \Cref{completing_the_covering_from_a_large_subset}.
\end{proof}

\subsection{Diameters of non-split forms}

\begin{theorem} \label[theorem]{theorem:diameter_bound_non_split}
    For any $\epsilon > 0$, there is a set of primes with density at least $1 - \epsilon$ such that the following holds. Let $\gfrak(\F_{p^d})^\Theta$ be a non-split form of a classical Lie algebra over $\F_p$. Then there is a constant $C_\epsilon > 0$, depending only on $\epsilon$ and the Lie type of $\gfrak(\F_{p^d})^\Theta$, such that for a uniformly random pair of elements $X, Y\in \gfrak(\F_{p^d})^\Theta$, we have
    \[
        \diam\left( \gfrak(\F_{p^d})^\Theta, \{X, Y\} \right) \leq C_\epsilon \log p
    \]
    with probability tending to $1$ as $p \to \infty$.
\end{theorem}

We rely again on the covering map from \Cref{lemma:covering_surjection}. Instead of covering all primes with a single lattice, we can pick an irreducible polynomial $f \in \Z[x]$ with a Galois number field $E = \Q[x]/(f(x))$ and form the Lie ring $\gfrak(\integers_E)^\Theta$. As long as the prime $p$ is such that $f \bmod p$ is irreducible, we can reduce $\gfrak(\integers_E)^\Theta$ modulo $p$ and obtain a non-split form over $\F_p$. After this, the same argument as in the proof of \Cref{theorem:diameter_bound_split} applies. There are two caveats here. The first one is that this argument only works for primes $p$ such that $f \bmod p$ is irreducible. We will address this by using several independent polynomials $f$ in order to cover as many primes as possible. The second caveat is that our argument involves the choice of a number field $E$, and the resulting constants $c, \delta$ in \Cref{lemma:initial_growth} depend on $E$. This is a more serious issue, as we consequentially cannot guarantee that the final constant $C_\epsilon$ in the statement of the theorem is uniform over an infinite collection of fields $E$ (see \Cref{remark:d2_need_infinitely_many_fields}) and therefore need to restrict to using only a finite collection. We now show how to construct such a finite collection of fields $E$.

\subsubsection{Chebotarev density theorem}

Let us first recall the classical theorem that allows us to control the density of primes modulo which a single polynomial is irreducible.

\begin{theorem}[Chebotarev density theorem]
Let $K$ be a Galois number field with $G = \Gal(K/\Q)$, and let $C \subseteq G$ be a conjugacy closed set. Then the set of primes $p$ such that $K/\Q$ is unramified at $p$ and whose associated Frobenius conjugacy class $\Frob_p(K/\Q)$ is contained in $C$ has natural density $\abs{C}/\abs{G}$.
\end{theorem}

We will use the theorem with $K = \Q[x]/(f(x))$, where $f \in \Z[x]$ is an irreducible polynomial. Let $\Delta$ be the discriminant of $f$. If a prime $p$ does not divide $\Delta$, then $K/\Q$ is unramified at $p$, and the Frobenius conjugacy class $\Frob_p(K/\Q)$ is a well-defined subset of $G = \Gal(K/\Q)$. Letting $d = \deg f$, we can identify $G$ as a subgroup of $\mathrm{Sym}(d)$ acting on the $d$ roots of $f$. The Frobenius conjugacy class $\Frob_p(K/\Q)$ is represented by a $d$-cycle in $\mathrm{Sym}(d)$ if and only if $f \bmod p$ is irreducible in $\F_p[x]$. Let $C^{\irr}$ be the set of elements of $G$ that act as a $d$-cycle on the roots of $f$. Then the Chebotarev density theorem gives us
\[
    \density \left\{ p \in \PP \mid 
    f \bmod p \text{ is irreducible} 
    \right\} 
    = \frac{\abs{C^{\irr}}}{\abs{G}}.
\]
We will require a version of this property for several independent polynomials. The precise independence condition needed here is as follows: we say that number fields $K_1, \dots, K_N$ are \emph{jointly independent} if, for each $i$, the compositum of the first $i-1$ fields intersects trivially with the $i$-th field, \emph{i.e.},
\[
(K_1 \cdots K_{i-1}) \cap K_i = \Q.
\]
In such a situation, an easy proof by induction shows that the restriction map
\[
\Gal(K_1 \cdots K_N/\Q) \to \Gal(K_1/\Q) \times \cdots \times \Gal(K_N/\Q)
\]
is an isomorphism. Under these assumptions, we can now state the version of the Chebotarev density theorem for several polynomials.

\begin{theorem}
Let $f_1, f_2, \dots,f_N \in \Z[x]$ be irreducible polynomials with jointly independent Galois number fields $K_i = \Q[x]/(f_i(x))$. 
Let
\[
X_i = \left\{ p \in \PP \mid
f_i \bmod p \text{ is irreducible} \right\}.
\] 
Then
\[
\density \left( \bigcup_{i = 1}^N X_i \right)
= 1 - \prod_{i=1}^N \left( 1 - \density(X_i) \right).
\]
\end{theorem}
\begin{proof}
Let $G_i = \Gal(K_i/\Q) \leq \mathrm{Sym}(d_i)$, where $d_i = \deg f_i$. Writing $C_i^{\irr}$ for the set of elements of $G_i$ that act as a $d_i$-cycle on the roots of $f_i$, we have
\[
\density(X_i) = \frac{\abs{C_i^{\irr}}}{\abs{G_i}},
\]
Since the number fields $K_i$ are jointly independent, their compositum $K = K_1 \cdots K_N$ with Galois group $G = \Gal(K/\Q)$ satisfies $G = G_1 \times \cdots \times G_N$. For each prime $p$ that does not divide any of the discriminants of $f_i$, the Frobenius conjugacy class $\Frob_p(K/\Q)$ can be identified with the tuple
\[
\Frob_p(K/\Q) = \left( \Frob_p(K_1/\Q), \dots, \Frob_p(K_N/\Q) \right) \in G_1 \times G_2 \times \cdots \times G_N = G.
\]
Hence all polynomials $f_i$ are reducible mod $p$ if and only if
\[
\Frob_p(K/\Q) \subseteq (G_1 \setminus C_1^{\irr}) \times \cdots \times (G_N \setminus C_N^{\irr}).
\]
By the Chebotarev density theorem, the set of primes $p$ for which this occurs has density
\[
\prod_{i = 1}^N \left( 1 - \frac{\abs{C_i^{\irr}}}{\abs{G_i}} \right)
= \prod_{i = 1}^N \left( 1 - \density(X_i) \right). \qedhere
\]
\end{proof}    

\subsubsection{Constructing suitable fields}

Finally it remains to construct the polynomials $f_i$ that will give us the jointly independent Galois number fields $K_i$. In our application, we only need to construct families of polynomials of degree $d = 2$ and $d = 3$ (depending on the order of the corresponding automorphism of the Dynkin diagram). Our construction will be based on the following lemma.

\begin{lemma}
Let $K_1, \dots, K_N$ be Galois number fields with coprime discriminants. Then the fields $K_i$ are jointly independent.
\end{lemma}
\begin{proof}
Induction on $N$. The base case $N = 1$ is trivial. Assume that $K_1,\dots, K_{N-1}$ are jointly independent. Ramification is determined by triviality of the inertia subgroup, which is multiplicative in the setting of joint independence \cite[Chapter 14]{ribenboim2013classical}. Hence primes that ramify in $K_1 \cdots K_{N-1}$ are precisely those that divide one of the discriminants of $K_1, \dots, K_{N-1}$. On the other hand, if a prime ramifies in $(K_1\cdots K_{N-1}) \cap K_N$, it also ramifies in both $K_1 \cdots K_{N-1}$ and $K_N$, and so it must divide the discriminant of $K_N$ as well as the discriminant of some $K_i$ with $i < N$. Since these are coprime, no primes ramify in the intersection $(K_1 \cdots K_{N-1}) \cap K_N$, so the intersection is trivial.
\end{proof}

We now show how to construct the jointly independent fields $K_i$ with coprime discriminants by using cyclotomic extensions. For each prime $q_i$ that is $1 \pmod{3}$, let $\Q(\zeta_{q_i})$ be a cyclotomic extension, where $\zeta_{q_i}$ is a primitive $q_i$-th root of unity. The Galois group of this extension is cyclic of order $q_i - 1$. This group has an index $d$ subgroup ($d = 2,3$), call it $H_i$. Let $K_i = \Q(\zeta_{q_i})^{H_i}$ be the corresponding fixed subfield. The degree $K_i$ over $\Q$ is $d$. Note that the discriminants of these fields are coprime: the discriminant of the cyclotomic extension $\Q(\zeta_{q_i})$ is a power of $q_i$ by the conductor-discriminant formula, and so the discriminant of $K_i$ is also a power of $q_i$. Finally, we can realize the extension $K_i$ as the splitting field of some polynomial $f_i \in \Z[x]$ of degree $d$. 

\begin{proof}[Proof of \Cref{theorem:diameter_bound_non_split}]
Use the polynomials $f_i$ to run the multipolynomial Chebotarev density theorem from above. Taking a sufficiently large $N$, we obtain polynomials $f_i$ such that the set of primes $p$ for which some $f_i \bmod p$ is irreducible has density $1 - \epsilon$. For each such prime $p$, we can take the corresponding Galois number field $K_i = \Q[x]/(f_i(x))$ and form the Lie ring $(\gfrak(\Z) \otimes \integers_{K_i})^\Theta$. The covering map from \Cref{lemma:covering_surjection} is then a surjection onto the non-split form $\gfrak(\F_{p^d})^\Theta$. The rest of the proof is identical to that of \Cref{theorem:diameter_bound_split}, using the generating pair from \Cref{generating_pair_in_covering_ring}. For each field $K_i$, we obtain a constant $C_i$ such that the diameter of $\gfrak(\F_{p^d})^\Theta$ with respect to a random pair is at most $C_i \log p$ with probability tending to $1$ as $p \to \infty$. We can thus take $C_\epsilon = \max_i C_i$ to complete the proof.
\end{proof}

\begin{remark} \label{remark:d2_need_infinitely_many_fields}
It is not possible to cover a density $1$ set of primes using only finitely many fields $E$ in the above argument. This obstruction is already visible for $d = 2$. In that case the fields $E$ are of the form $\Q(\sqrt{D})$, and a prime $p$ coprime to $D$ is inert in $E$ precisely when $D$ is a quadratic non-residue modulo $p$. Now let $n\geq 2$ be fixed. We claim that there is a positive density set of primes $p$ for which every integer $D$ with $2\leq D\leq n$ is a quadratic residue modulo $p$. It is enough to impose this condition for the primes $q\leq n$. Let $K_q = \Q(\sqrt q)$, the splitting field of $f_q = x^2 - q$. As $q \leq n$ ranges over the primes, the fields $K_q$ are jointly independent. By the multipolynomial Chebotarev density theorem above, the density of primes $p$ for which none of the $f_q$ is irreducible mod $p$ is $2^{-\pi(n)}$, where $\pi(n)$ is the number of primes less than or equal to $n$. For every such prime $p$, each prime $q\leq n$ is a residue modulo $p$. In other words, the smallest positive quadratic non-residue modulo $p$ is larger than $n$. Thus, for every fixed $n$, the quadratic fields $\Q(\sqrt D)$ with $2\leq D\leq n$ miss a positive density set of primes.
\end{remark}

\appendix

\section{Extremal bases of forms}
\label{sec:extremal_bases}

In this section, we complete the technical proof of \Cref{prop:extremal_basis} by constructing an extremal basis for each classical Lie algebra form $\gfrak(E)^\Theta$ over a field $E = \F_{p^d}$ and verifying that it satisfies the required quadratic condition. All computations were performed symbolically using Wolfram Mathematica and are available in the repository \cite{github-extremal}. We outline the general strategy below.

We begin the construction with a highest weight vector $x \in \gfrak(E)^\Theta$ and the corresponding lowest weight vector $y$.\footnote{The vectors are stored as symbolic expressions in Mathematica. For example, in the case of $\slfrak_3(\F_{p^2})^\Theta$, we take $x = (\omega - \omega^\sigma) E_{13}$, where $\omega$ and $\omega^\sigma$ are independent commuting variables. All calculations are then performed over the field $\Q(\omega, \omega^\sigma)$.} Let $h = [x, y]$. Thus, $\{ x, h, y \}$ form an $\slfrak_2(\F_p)$-triple. Following \cite[Proposition 2.4]{cohen2008simple}, the vector space $\gfrak(E)^\Theta$ decomposes into a direct sum of subspaces $L_{-2}, L_{-1}, L_0, L_1, L_2$, where $L_i$ is the eigenspace of $-\ad_h$ with eigenvalue $i$. For any $z \in L_1$, let $u(z) = \exp(\ad_z)x$. This element is well defined as long as $p > 3$, since $\ad_z$ is nilpotent of order at most $5$. By \cite[Proposition 4.1]{cohen2008simple}, the elements $x$, $y$, and $u(z)$, as $z$ ranges over a basis of $L_1$, form an extremal generating set $\mathcal{E}$ of $\gfrak(E)^\Theta$.

To obtain an extremal basis from this generating set, we follow \cite[Lemma 2.5]{cohen2001lie}. Given two extremal elements $a$ and $b$, we define the new element $\eta(a,b) = \exp(\ad_a)b$. This element is again extremal, and $[a,b]$ lies in the span of $a$, $b$, and $\eta(a,b)$. By repeatedly applying the operation $\eta(a,b)$ to pairs of elements from $\mathcal{E}$, we obtain a set of extremal elements $\mathcal{E}'$. Continuing this process, and applying the operation $\eta(a,b)$ to pairs $a \in \mathcal{E}$, $b \in \mathcal{E}'$, we obtain a larger set $\mathcal{E}''$. Iterating the $\eta$ construction as needed, we eventually obtain a set $\mathcal{E}^\infty$ that spans $\gfrak(E)^\Theta$. The following example illustrates this process.

\begin{example}
Let $\gfrak = \slfrak_n$ with $n$ even and $\Theta=\vartheta \sigma$ of order $2$. We can take $x = E_{1n}$ and $y = E_{n1}$. The space $L_1$ is spanned by $Z_1(i) = E_{i1} + (-1)^i E_{n,n+1-i}$ and $Z_2(i) = \omega E_{i1} + (-1)^i  \omega^\sigma E_{n,n+1-i}$ for all $2 \leq i \leq n-1$. We then construct the elements
\begin{align*}
    U_1(i) &= E_{1, n} + (-1)^{1 + i} E_{1, n+1-i} + E_{i, n} + (-1)^{1 + i} E_{i, n+1-i}, \\
    U_2(i) &= E_{1, n} + (-1)^{1 + i} \omega^\sigma E_{1, n+1-i} + \omega (E_{i, n} + (-1)^{1 + i} \omega^\sigma E_{i, n+1-i}).
\end{align*}
The set $\mathcal{E} = \{ x, y \} \cup \{ U_1(i), U_2(i) \mid 2 \leq i \leq n-1 \}$ is an extremal generating set of $\slfrak_n(E)^\Theta$. The larger set $\mathcal{E}''$ contains the following basis of $\slfrak_n(E)^\Theta$:
\begin{align*}
    \mathcal{X} =\; & \{ x,\, y,\, \eta(x,y) \} 
    \cup \{ U_1(i),\, U_2(i), \, \eta(U_1(i), y),\, \eta(U_2(i), y) \mid 2 \leq i \leq n-1 \}  \\
    & \cup \left\{ \eta(U_1(i), \eta(x,y)),\, \eta(U_2(i), \eta(x,y)) \mid 2 \leq i \leq n-1 \right\} \\
    & \cup \left\{ \eta(U_1(i), U_1(n+1-i)),\, \eta(U_2(i), U_2(n+1-i)) \mid 2 \leq i \leq n/2 \right\} \\
    & \cup \left\{ 
        \eta(U_1(j), \eta(U_1(i), y)),\, 
        \eta(U_1(j), \eta(U_2(i), y)) 
        \mid i \neq j,\, i+j \neq n+1 
    \right\}.
\end{align*}
\end{example}

It remains to verify that we can always choose a basis $\mathcal{X}$ of $\gfrak(E)^\Theta$ from $\mathcal{E}^\infty$ that satisfies the quadratic condition from \Cref{prop:extremal_basis}. We show that this can be achieved by taking $a = y$. In most cases, the corresponding element $z$ can simply be taken as either $x$ or $[b, y]$ for any $b \in \mathcal{X} \setminus \{ y \}$.\footnote{Interestingly, there may exist elements $b \in \mathcal{E}^\infty$ for which the quadratic condition is \emph{not} satisfied with $a = y$, so care must be taken in the choice of $\mathcal{X}$. This occurs, for example, in the case of non-split $A_n$ for the element $b = \eta(U_1(i), U_1(n+1-i))$ with even $i$ from the example above.} We construct $\mathcal{X}$ either by explicitly specifying its elements and verifying the quadratic condition, or, in the non-split $D_4$ and $E_6$ cases, by showing that $\mathcal{E}''$ spans $\gfrak(E)^\Theta$ and that a basis $\mathcal{X}$ can be selected from it so that taking $z = x$ or $z = [b, y]$ suffices. In the $D_4$ case, matrix rank is computed over the field $\Q$ extended by commuting variables $\omega$, $\omega^\sigma$, and $\omega^{\sigma^2}$. The rank does not drop when passing to a finite field for sufficiently large $p$, as ensured by the lemma below.

\begin{lemma}
Let $d$ be a prime. Let $P$ be a nonzero polynomial with integer coefficients in $d$ variables of total degree $N$. For any prime $p > N + 1$, there exists $\omega \in \F_{p^d} \setminus \F_p$ so that $P(\omega, \omega^p, \dots, \omega^{p^{d-1}}) \neq 0$ in $\F_{p^d}$.
\end{lemma}
\begin{proof}
Let $Q(x) = P(x, x^p, \dots, x^{p^{d-1}})$. Then $Q$ is a nonzero polynomial in one variable of degree at most $N p^{d-1}$, hence it has at most as many roots in $\F_{p^d}$. The hypothesis $p>N+1$ ensures that this number is less than the number of elements in $\F_{p^d} \setminus \F_p$. Therefore, we can find $\omega$ as desired.
\end{proof}

\bibliographystyle{alpha}
\bibliography{references}

\end{document}